\documentclass{amsart}

\usepackage{amsmath}
\usepackage{amsthm,thmtools}
\usepackage{amssymb}
\usepackage{tikz-cd}
\usepackage{hyperref}
\usepackage{enumitem}
\usepackage{tabularx}
\usepackage{array}
\usepackage[english]{babel}
\usetikzlibrary{babel}

\newcommand{\NN}{\mathbb{N}}
\newcommand{\CC}{\mathbb{C}}

\let\SS\relax
\newcommand{\SS}{\mathbb{S}}
\newcommand{\ZZ}{\mathbb{Z}}

\newcommand{\Acal}{\mathcal{A}}

\newcommand{\Ccal}{\mathcal{C}}
\newcommand{\Dcal}{\mathcal{D}}
\newcommand{\Ecal}{\mathcal{E}}
\newcommand{\Fcal}{\mathcal{F}}

\newcommand{\Scal}{\mathcal{S}}

\DeclareMathOperator{\Hom}{Hom}

\newcommand{\spec}{\mathrm{Spec}}

\newcommand{\Sets}{\mathbf{Sets}}

\newcommand{\Sh}{\mathbf{Sh}}
\newcommand{\PSh}{\mathbf{PSh}}

\newtheorem{definition}{Definition}[section]
\newtheorem{proposition}[definition]{Proposition}

\newtheorem{theorem}[definition]{Theorem}
\newtheorem{corollary}[definition]{Corollary}
\newtheorem{lemma}[definition]{Lemma}
\newtheorem{example}[definition]{Example}
\newtheorem{remark}[definition]{Remark}

\tikzset{
b/.style={bend left=10},
bb/.style={bend left},
cl/.style={outer sep=-1pt},
}

\let\phi\varphi
\let\theta\vartheta

\newcommand\rurl[1]{%
  \href{http://#1}{\nolinkurl{#1}}%
}
\newcommand\rsurl[1]{%
  \href{https://#1}{\nolinkurl{#1}}%
}

\AtBeginDocument{%
   \def\MR#1{}
}

\title{Toposes over which essential implies locally connected}
\author{Jens Hemelaer}
\address{Department of Mathematics, University of Antwerp \\ 
 Middelheimlaan 1, B-2020 Antwerp (Belgium) \\ {\tt jens.hemelaer@uantwerpen.be}}

\DeclareMathOperator{\bc}{\triangleright}
\DeclareMathOperator{\bcp}{\underline{\triangleright}}

\begin{document}

\begin{abstract}
We introduce the notion of an EILC topos: a topos $\Ecal$ such that every essential geometric morphism with codomain $\Ecal$ is locally connected. We then show that the topos of sheaves on a topological space $X$ is EILC if $X$ is Hausdorff (or more generally, if $X$ is Jacobson). Further examples of Grothendieck toposes that are EILC are Boolean \'etendues and classifying toposes of compact groups. Next, we introduce the weaker notion of CILC topos: a topos $\Ecal$ such that any geometric morphism $f : \Fcal \to \Ecal$ is locally connected, as soon as $f^*$ is cartesian closed. We give some examples of topological spaces $X$ and small categories $\Ccal$ such that $\Sh(X)$ resp.\ $\PSh(\Ccal)$ are CILC. Finally, we show that any Boolean elementary topos is CILC.
\end{abstract}

\maketitle

\tableofcontents

\section{Introduction}

For elementary toposes $\Ecal$ and $\Fcal$, a geometric morphism $f : \Fcal \to \Ecal$ is called \textbf{essential} if the inverse image functor $f^*$ has a left adjoint, that is then usually written as $f_!$. Moreover, we say that $f$ is \textbf{locally connected} (or molecular) if  $f^*$ has an $\Ecal$-indexed left adjoint, or equivalently, if $f$ is essential and the natural morphism
\begin{equation} \label{eq:frobenius}
f_!(X \times_{f^*B} f^*A) \to f_!(X) \times_B A
\end{equation}
is an isomorphism, for all morphisms $A \to B$ in $\Ecal$ and $X \to f^*B$ in $\Fcal$, see \cite{barr-pare}. The notion of a locally connected geometric morphism is more natural from a geometric point of view; in particular, locally connected geometric morphisms are stable under base change, while essential geometric morphisms are not.

In this article, we will follow an idea of Mat\'ias Menni, formulated in his message ``Essential vs Molecular'' on the category theory mailing list (May 3, 2017), where he mentioned the problem of characterizing the toposes $\Ecal$ such that any essential geometric morphism $f : \Fcal \to \Ecal$ is locally connected. Elementary toposes with this property will here be called EILC (``Essential Implies Locally Connected''). We will show that surprisingly many toposes are EILC, including for example the topos of sheaves $\Sh(X)$ for $X$ an arbitrary Hausdorff topological space.

The notion of an EILC topos has applications to the study of levels of an elementary topos $\Ecal$. A \textbf{level} of $\Ecal$, as introduced by Lawvere, is by definition a subtopos $\Ecal'$ of $\Ecal$ such that the inclusion geometric morphism $i : \Ecal' \to \Ecal$ is essential. For example, each open subtopos of $\Ecal$ defines a level of $\Ecal$. Conversely, if $\Ecal$ is an EILC topos, then for any level $\Ecal'$, the inclusion $i : \Ecal' \to \Ecal$ must be locally connected, and because locally connected geometric morphisms are open, we find that any level of $\Ecal$ is given by an open subtopos. So for EILC toposes, the structure of the levels is completely known.

Another situation where EILC toposes are relevant, and the original motivation for this paper, is in the study of precohesive geometric morphisms. In \cite{lawvere-axiomatic}, Lawvere introduced an axiomatic setting for when a category $\Ecal$ can be seen as a ``category of spaces'' over a base category $\Scal$, with both $\Ecal$ and $\Scal$ cartesian closed and extensive. A first requirement is that there is a string of adjoint functors
\begin{equation*}
\begin{tikzcd}[column sep = huge]
\Ecal \ar[rr,bend left=40,"{f_!}"] \ar[rr,"{f_*}"] & & 
\Scal \ar[ll,bend right=20,"{f^*}"'] \ar[ll,bend left=20,"{f^!}"']
\end{tikzcd} \qquad \quad f_! \dashv f^* \dashv f_* \dashv f^!
\end{equation*} 
between $\Ecal$ and $\Scal$.
Here $f^*$ is thought of as the functor that sends an object in $\Scal$ to its associated discrete space object in $\Ecal$. Then for $X$ in $\Ecal$, $f_!(X)$ has an interpretation as the object of connected components (or ``pieces'') of $X$, and $f_*(X)$ can be thought of as the object of points of $X$. 

Further relevant axioms in this setting are that $f^*$ (or equivalently, $f^!$) is fully faithful, that $f_!$ preserves finite products, and that the natural map $f_* \to f_!$ is an epimorphism. Lawvere calls this last condition the Nullstellensatz: it expresses that each component has at least one point. If all the conditions above are satisfied, then $\Ecal$ is said to be \textbf{precohesive} over $\Scal$, see the work of Lawvere and Menni \cite[Definition 2.4]{lawvere-menni}. 

A particular case of interest is when the string of adjoint functors $f_! \dashv f^* \dashv f_* \dashv f^!$ arises from a geometric morphism $f : \Ecal \to \Scal$ between elementary toposes (with $f^*$ the inverse image functor). Because $f^*$ has a left adjoint $f_!$, the geometric morphism $f$ is necessarily essential. Moreover, recall that by definition a geometric morphism $f$ is \textbf{local} if and only if $f^*$ is fully faithful and $f_*$ has a further right adjoint $f^!$. Finally, the Nullstellensatz holds whenever $f$ is hyperconnected \cite[Lemma 3.1]{lawvere-menni}. As a result, $\Ecal$ is precohesive over $\Scal$ if and only if $f$ is hyperconnected, essential and local, with $f_!$ preserving finite products. If this is the case, then the geometric morphism $f$ is itself called precohesive.

A natural question is now whether precohesive geometric morphisms $f : \Ecal \to \Scal$ are stable under \'etale base change, or in other words whether for an object $X$ in $\Scal$, the induced geometric morphism on slice toposes
\begin{equation*}
f/X : \Ecal/f^*(X) \to \Scal/X
\end{equation*}
is again precohesive. It was shown in \cite[Corollary 10.4]{lawvere-menni} that this is the case whenever $f$ is locally connected. A question that was left open in \cite{lawvere-menni} is then whether already every precohesive geometric morphism is locally connected. An affirmative answer would be useful in practice: it is often difficult to verify explicitly whether the map \eqref{eq:frobenius} is an isomorphism. 

At the moment, there seem to be no heuristic arguments to support the idea that all precohesive geometric morphisms are locally connected. So to settle the problem, it is natural to focus on the construction of an example of a precohesive geometric morphism that fails to be locally connected. Recently, there have been renewed efforts to finding such an example. In \cite{not-locally-connected}, Morgan Rogers and the present author constructed an essential, hyperconnected, local geometric morphism that is not locally connected, answering a question by Thomas Streicher in the category theory mailing list (``Does essential entail locally connected for hyperconnected geometric morphisms?'', September 19, 2020). The constructed geometric morphism is however not precohesive, because $f_!$ does not preserve finite products. In another direction, Garner and Streicher in \cite{garner-streicher} constructed essential, local geometric morphisms $f$, with $f_!$ preserving finite products, such that $f$ fails to be locally connected. Again, $f$ is not precohesive (it is not hyperconnected). Note that in examples like this, the codomain can not be an EILC topos. So whenever we show that a certain type of toposes is EILC, these toposes can be added to the list of codomain toposes to avoid when constructing examples as above. 

The main goal of the present article is to show that the topos of sheaves $\Sh(X)$ on a topological space is EILC if $X$ is Jacobson. Here we say that a topological space is \textbf{Jacobson} if two open subsets are equal whenever they contain the same closed points, see e.g.\ \cite[\href{https://stacks.math.columbia.edu/tag/005T}{Section 005T}]{stacks-project}. For $T_1$ topological spaces (in particular, Hausdorff topological spaces) this condition is automatically satisfied, because in this case all points are closed. Further, the spectrum $\spec(R)$ of a commutative ring $R$ is Jacobson (for the Zariski topology) if and only if $R$ is a Jacobson ring. As a result, there are many examples of Jacobson spaces that are not Hausdorff, for example $\spec(\ZZ)$ or $\spec(\CC[x,y])$.

In order to give a more comprehensive list of EILC toposes, we also show in Section \ref{sec:boolean-etendues} that a Grothendieck topos is EILC if it is a Boolean \'etendue, or if it is a classifying topos of a compact topological group. In particular, the petit \'etale topos of a field is EILC, because it coincides with the classifying topos of the absolute Galois group of the field (which is compact). An intriguing problem that is left open is whether the petit \'etale topos of a Jacobson ring is EILC. 

In the last section, we introduce the more general class of CILC toposes (``Cartesian closed Implies Locally Connected''). These are the elementary toposes $\Ecal$ such that any geometric morphism $f : \Fcal \to \Ecal$ is locally connected, as soon as $f^*$ is cartesian closed (i.e.\ preserves exponential objects). We then introduce a notion of weakly Jacobson geometric morphism, and we show that if $f : \Ecal \to \Scal$ is weakly Jacobson and $\Scal$ is EILC, then $\Ecal$ is CILC (under the assumption that $\Ecal$ has a natural number object). Further, we give a characterization of topological spaces $X$ and small categories $\Ccal$ such that $\Sh(X)$ resp.\ $\PSh(\Ccal)$ are weakly Jacobson over the topos of sets. Finally, we show that all Boolean elementary toposes are CILC, extending an earlier result by Mat\'ias Menni, who showed that if $\Scal$ is a Boolean topos and $f : \Ecal \to \Scal$ is a connected essential geometric morphism with $f_!$ preserving products, then $f$ is locally connected.

Now let us return to the motivating problem: finding an example of a precohesive geometric morphism $f : \Ecal \to \Scal$ that fails to be locally connected. As we saw above, $\Scal$ can not be EILC in such an example. Because precohesive geometric morphisms have cartesian closed inverse image functor $f^*$, we can even say that $\Scal$ can not be CILC. From the examples in this article, we conclude that $\Scal$ can not be a Boolean topos (as already shown in a different way by Menni), or a topos of sheaves on a Jacobson space, the topos of sheaves on the Sierpinski space or the topos of presheaves on a commutative monoid, among others. Having a list of codomain toposes to avoid might help to eventually find an example, if it exists.

\section{Background on Beck--Chevalley conditions}

\begin{definition}
We write $g \bc^q_p f$ if there is a commutative diagram
\begin{equation} \label{eq:bc-diagram}
\begin{tikzcd}
\Fcal' \ar[r,"{q}"] \ar[d,"{g}"'] & \Fcal \ar[d,"{f}"] \\
\Ecal' \ar[r,"{p}"] & \Ecal 
\end{tikzcd}
\end{equation}
such that the natural map
\begin{equation*}
f^* p_* \to q_* g^*
\end{equation*}
is an isomorphism (the Beck--Chevalley condition). Further, we write $g \bc_p f$ if there exists a morphism $q$ with $g \bc_p^q f$, and $g \bcp_p f$ if moreover $q$ can be chosen such that \eqref{eq:bc-diagram} is a pullback diagram.
\end{definition}

In order for pullbacks of elementary toposes to exist, we need some technical conditions. Recall that a geometric morphism $f : \Fcal \to \Ecal$ is \textbf{bounded} if there is an object $B$ in $\Fcal$ such that for every object $X$ in $\Fcal$, there is an object $I$ in $\Ecal$ such that $X$ is a subquotient of $B \times p^*(I)$, see \cite[Definition B3.1.7]{johnstone-elephant}. The pullback of two geometric morphisms $f : \Fcal \to \Ecal$ and $p : \Ecal' \to \Ecal$ exists if either $f$ or $p$ is bounded \cite[Proposition B3.3.6]{johnstone-elephant}. All localic geometric morphisms are bounded \cite[Examples B3.1.8(a)]{johnstone-elephant}. In particular, inclusions and \'etale geometric morphisms are bounded.

When discussing Beck--Chevalley conditions in topos theory, the notion of a tidy geometric morphism is relevant:

\begin{definition}
Let $p : \Ecal' \to \Ecal$ be a geometric morphism. Then we say that $p$ is \textbf{tidy} if $p_*$ preserves filtered $\Ecal$-indexed colimits. 
\end{definition}

For an extensive treatment of tidy geometric morphisms, see e.g.\ Moerdijk and Vermeulen \cite[Chapter III]{moerdijk-vermeulen} or Johnstone \cite[C3.4]{johnstone-elephant}

We recall some of the history behind this concept, following the introduction of \cite{moerdijk-vermeulen}. The concept of a tidy geometric morphism was first studied by Edwards in her PhD thesis \cite{edwards-thesis}, in the special case where the codomain topos is $\Sets$. Later, the concept was introduced for an arbitrary codomain topos by Tierney, and developed by Lindgren in his PhD thesis \cite{lindgren-thesis}. Lindgren referred to these geometric morphisms as being ``proper''. Moerdijk and Vermeulen later used the name ``tidy'' instead, to distinguish the concept from a notion of properness as introduced by Johnstone.

In practice, it might be difficult to check whether a given geometric morphism is tidy. However, every closed inclusion is tidy \cite[Chapter III, Corollary 5.8]{moerdijk-vermeulen}, so this gives a large family of concrete examples.

We recall the following properties from the literature.

\begin{proposition}[{See \cite{johnstone-elephant}}] \label{prop:locally-connected-tidy} Consider a pullback diagram
\begin{equation*}
\begin{tikzcd}
\Fcal' \ar[r,"{q}"] \ar[d,"{g}"'] & \Fcal \ar[d,"{f}"] \\
\Ecal' \ar[r,"{p}"] & \Ecal
\end{tikzcd}\quad.
\end{equation*}
\begin{enumerate}
\item If $f$ is locally connected and $p$ is bounded, then $g \bcp_p f$. 
\item If $f$ is bounded and locally connected, then $g \bcp_p f$. 
\item If $p$ is bounded and tidy, and $\Ecal$ has a natural number object, then $g \bcp_p f$.
\item If $p$ is bounded and tidy and $f$ is bounded, then $g \bcp_p f$.
\end{enumerate}
\end{proposition}

A proof for $(1)$ and $(2)$ is given in \cite[Theorem C3.3.15]{johnstone-elephant}. Further, $(3)$ corresponds to \cite[Theorem C3.4.7]{johnstone-elephant} and $(4)$ corresponds to \cite[Theorem C3.4.10]{johnstone-elephant}. See also \cite[Chapter III, Theorem 4.8]{moerdijk-vermeulen}, where the boundedness assumption is implicit. If we restrict to Grothendieck toposes, then all geometric morphisms are automatically bounded, and moreover every Grothendieck topos has a natural object. In this setting, (3) and (4) coincide and are attributed to Lindgren \cite{lindgren-thesis}.

Beck--Chevalley squares can be pasted in the following way:

\begin{proposition}[Transitivity] \label{prop:transitivity}
If $h \bc_{p'}^{q'} g$ and $g \bc_p^q f$, then $h \bc_{pp'}^{qq'} f$.
\end{proposition}
\begin{proof}
Consider the commutative diagram
\begin{equation*}
\begin{tikzcd}
\Fcal'' \ar[r,"{q'}"] \ar[d,"{h}"'] & \Fcal' \ar[r,"{q}"] \ar[d,"{g}"] & \Fcal \ar[d,"{f}"] \\
\Ecal'' \ar[r,"{p'}"] & \Ecal' \ar[r,"{p}"] & \Ecal 
\end{tikzcd}
\end{equation*}
Then it follows that $f^* p_* p'_* \simeq q_* g^* p'_* \simeq q_* q'_* h^*$. 
\end{proof}

We now introduce the following definition:

\begin{definition}
Let $f : \Fcal \to \Ecal$ and $p : \Ecal' \to \Ecal$ be geometric morphisms. We say that $f$ is \textbf{locally connected at $p$} if there is a locally connected geometric morphism $g$ such that $g \bc_p f$.
\end{definition}

\begin{proposition}[Descent] \label{prop:descent}
Let $f : \Fcal \to \Ecal$ be an essential geometric morphism, and let $\{ p_i : \Ecal_i \to \Ecal \}_{i\in I}$ be a jointly surjective family of geometric morphisms. If $f$ is locally connected at $p_i$ for each $i \in I$, then $f$ is locally connected.
\end{proposition}
\begin{proof}
It is enough to show that the map
\begin{equation*}
\theta : f_!(X \times_{f^*B} f^*A) \to f_!(X) \times_B A
\end{equation*}
is an isomorphism, for each $X$ in $\Fcal$ and each diagram $f_!(X) \to B \leftarrow A$ in $\Ecal$. Because the family $\{p_i\}_{i \in I}$ is jointly surjective, it is enough to prove that each $p_i^*(\theta)$ is an isomorphism.

Take $g$ and $q$ such that $fq = p_ig$, with $g$ locally connected, and such that the natural map $f^*p_{i,*} \to q_*g^*$ is an isomorphism. Because $f$ and $g$ are essential, there is also a natural isomorphism $g_!q^* \to p_i^*f_!$.
We compute:
\begin{align*}
p_i^*f_!(X \times_{f^*B} f^*A) &\simeq g_!(q^*X \times_{q^*f^*B} q^*f^*(A)) \\
&\simeq g_!(q^*X \times_{g^*p_i^*B} g^*p_i^*A ) \\
&\simeq g_!(q^*X) \times_{p_i^*B} p_i^*A \\
&\simeq p_i^*f_!(X) \times_{p_i^*B} p_i^*A \\
&\simeq p_i^*( f_!(X) \times_B A )
\end{align*}
where in the third isomorphism we use that $g$ is locally connected.
\end{proof}

\begin{proposition}[Stability] \label{prop:stability}
Suppose that $g \bcp_p f$ with $p$ an inclusion. If $f$ is essential, then $g$ is essential as well.
\end{proposition}
\begin{proof}
We write $q$ for the pullback of $p$ along $f$, so we have a pullback diagram of the form
\begin{equation}
\begin{tikzcd}
\Fcal' \ar[r,"{q}"] \ar[d,"{g}"'] & \Fcal \ar[d,"{f}"] \\
\Ecal' \ar[r,"{p}"] & \Ecal 
\end{tikzcd}
\end{equation}
We claim that $p^*f_!q_*$ is a left adjoint for $g^*$. We compute:
\begin{align*}
\Hom_{\Ecal'}(p^*f_!q_*X,Y) &\simeq \Hom_\Fcal(q_*X,f^*p_*Y) \\
&\simeq \Hom_{\Fcal}(q_*X, q_*g^*Y) \\
&\simeq \Hom_{\Fcal'}(X,g^*Y)
\end{align*}
where in the second natural bijection we use the Beck--Chevalley condition, and in the third natural bijection we use that $q$ is an inclusion (as pullback of the inclusion $p$). It follows that $p^*f_!q_*$ is the left adjoint of $g^*$, so $g$ is essential.
\end{proof}

We will also need the following well-known characterization of cartesian closedness for the inverse image functor.

\begin{proposition}[Cartesian closedness] \label{prop:characterization-cartesian-closed}
Let $f : \Fcal \to \Ecal$ be a geometric morphism. Then the following are equivalent:
\begin{enumerate}
\item $f^*$ is cartesian closed;
\item $(f/E) \bcp_{\pi_E} f$ for every object $E$ in $\Ecal$, with $\pi_E : \Ecal/E \to \Ecal$ the \'etale geometric morphism corresponding to $E$ and $f/E$ the pullback of $f$ along $\pi_E$. 
\end{enumerate}
\end{proposition}
\begin{proof}
For $E$ in $\Ecal$, consider the pullback diagram
\begin{equation*}
\begin{tikzcd}
\Fcal/f^*(E) \ar[r,"{\tilde{\pi}_E}"] \ar[d,"{f/E}"'] & \Fcal \ar[d,"{f}"] \\
\Ecal/E \ar[r,"{\pi_E}"] & \Ecal 
\end{tikzcd}.
\end{equation*}
Both $f$ and $f/E$ are essential, so the Beck--Chevalley isomorphism is in this case given by
\begin{equation*}
(f/E)_!\tilde{\pi}_E^* \simeq \pi_E^*f_!.
\end{equation*}
This amounts to the condition that the morphism
\begin{equation*}
\theta_{F,E} : f_!(F \times f^*(E)) \to f_!(F) \times E
\end{equation*}
is an isomorphism, for each $F$ in $\Fcal$. This is precisely the Frobenius map, and $f^*$ is cartesian closed if and only if $\theta_{F,E}$ is an isomorphism for all $F$ and $E$ (see \cite[Lemma A1.5.8]{johnstone-elephant}). We conclude that $f^*$ is cartesian closed if and only if $(f/E)\bcp_{\pi_E} f$ for each object $E$ in $\Ecal$.
\end{proof}

\section{Jacobson topological spaces and Jacobson \'etendues}

An elementary topos $\Ecal$ will be called \textbf{EILC} if any essential geometric morphism $f : \Fcal\to\Ecal$ is locally connected. We will first show that $\Sh(X)$ is EILC for any Jacobson topological space $X$.

\begin{definition} \label{def:jacobson-space}
Let $X$ be a topological space, and let $X_0 \subseteq X$ be its subspace of closed points. We then say that $X$ is \textbf{Jacobson} if $U \cap X_0 = V \cap X_0$ implies $U = V$, for all open subsets $U,V \subseteq X$.
\end{definition}

Equivalently, $X$ is Jacobson if and only if for every closed subset $Z \subseteq X$, the subset $Z \cap X_0 \subseteq Z$ is dense, see \cite[\href{https://stacks.math.columbia.edu/tag/005T}{Section 005T}]{stacks-project}. However, Definition \ref{def:jacobson-space} is more natural from a topos-theoretic point of view: it says precisely that the closed points of $X$ define a jointly surjective family of points for $\Sh(X)$.

If $X$ is the spectrum of a commutative ring $R$, with the Zariski topology, then $X$ is Jacobson if and only if $R$ is a Jacobson ring, in the sense that each prime ideal is an intersection of maximal ideals \cite[\href{https://stacks.math.columbia.edu/tag/00G3}{Lemma 00G3}]{stacks-project}.

We can generalize the notion of Jacobson topological space over an arbitrary base elementary topos $\Scal$ as follows: we say that a localic geometric morphism $e : \Ecal \to \Scal$ is a \textbf{Jacobson space} if $e$ is localic and the family of closed points $p : \Scal \to \Ecal$ is jointly surjective (points are by definition sections of $e$, i.e.\ $ep\simeq 1$). If $e : \Ecal \to \Scal$ is a Jacobson space over $\Scal$, then in particular $\Ecal$ has enough points over $\Scal$. Note that if $e$ is localic, then any point $p : \Scal \to \Ecal$ is an inclusion. Indeed, we have a pullback of the form
\begin{equation*}
\begin{tikzcd}
\Scal \ar[r,"{p}"] \ar[d,"{p}"'] & \Ecal \ar[d,"{(pe,1_\Ecal)}"] \\
\Ecal \ar[r,"{\Delta}"] & \Ecal \times_\Scal \Ecal
\end{tikzcd}
\end{equation*}
and because $e$ is localic, the diagonal morphism $\Delta$ is an inclusion \cite[Proposition B3.3.8(ii)]{johnstone-elephant}. It then follows that its pullback $p$ is an inclusion as well. So, in this setting, $p : \Scal \to \Ecal$ is closed as geometric morphism (in the sense of \cite[C3.2, p.629]{johnstone-elephant}) if and only if $p$ defines a closed subtopos.

\begin{lemma} \label{lmm:EILC-from-closed-subtoposes}
Let $\Ecal$ be an elementary topos with a natural number object, and let $\{ p_i : \Ecal_i \to \Ecal \}_{i \in I}$ be a jointly surjective family, with each $p_i$ a closed inclusion and with each $\Ecal_i$ EILC. Then $\Ecal$ is EILC as well.
\end{lemma}
\begin{proof}
Take an essential geometric morphism $f : \Fcal \to \Ecal$. We will show that $f$ is locally connected. For each $p_i$, we consider the pullback diagram
\begin{equation*}
\begin{tikzcd}
\Fcal_i \ar[r,"{q_i}"] \ar[d,"{f_i}"'] & \Fcal \ar[d,"{f}"] \\
\Ecal_i \ar[r,"{p_i}"] & \Ecal 
\end{tikzcd}.
\end{equation*}
Because $p_i$ is a closed inclusion, it is in particular tidy \cite[Chapter III, Corollary 5.8]{moerdijk-vermeulen}. Further, any inclusion is bounded. So by Proposition \ref{prop:locally-connected-tidy}.(3) we have $f_i \bcp_{p_i} f$. It now follows from Proposition \ref{prop:stability} that $f_i$ is essential. Because $\Ecal_i$ is by assumption EILC, it follows that $f_i$ is locally connected. As a result, $f$ is locally connected at $p_i$, for each $i \in I$. Using Proposition \ref{prop:descent} and the fact that the family $\{p_i\}_{i \in I}$ is jointly surjective, we can then conclude that $f$ is locally connected.
\end{proof}

\begin{theorem} \label{thm:jacobson-space-is-EILC}
Let $X$ be a Jacobson topological space. Then $\Sh(X)$ is EILC.
\end{theorem}
\begin{proof}
Let $X$ be a Jacobson topological space, and let $X_0 \subseteq X$ be the subset of closed points. Then the family $\{ p_x : \Sets \to \Sh(X) \}_{x \in X_0}$ is a jointly surjective family of closed inclusions, where $p_x$ denotes the closed inclusion corresponding to the closed subset $\{x\} \subseteq X$. Further, $\Sets$ is EILC, so from Lemma \ref{lmm:EILC-from-closed-subtoposes} we deduce that $\Sh(X)$ is EILC as well.
\end{proof}

The proof above generalizes to a Jacobson space $f : \Ecal \to \Scal$, for $\Scal$ an EILC base topos, under the assumption that $\Ecal$ has a natural numbers object. In this case, by definition the family of closed points $\Scal \to \Ecal$ is jointly surjective, so as soon as $\Scal$ is EILC we can apply Lemma \ref{lmm:EILC-from-closed-subtoposes}.

We can also generalize the notion of Jacobson space in another direction as follows. Recall that an object $E$ of a topos $\Ecal$ is called \textbf{well-supported} if the unique morphism $E \to 1$ is an epimorphism. Further, an \textbf{\'etendue} is a topos $\Ecal$ such that there is a well-supported object $E$ in $\Ecal$ such that $\Ecal/E$ is localic (over the base topos $\Scal$).

\begin{definition}
Fix an elementary topos $\Scal$. A geometric morphism $e : \Ecal \to \Scal$ will be called a \textbf{Jacobson \'etendue} if there is a well-supported object $E$ in $\Ecal$ such that the composition 
\begin{equation*}
\Ecal/E \stackrel{\pi}{\to} \Ecal \stackrel{e}{\to} \Scal
\end{equation*}
is a Jacobson space. 
\end{definition}

If $e : \Ecal \to \Scal$ is a Jacobson space, then it is also a Jacobson \'etendue; in this case we can take $E=1$. 

\begin{example}
We give two examples of Grothendieck toposes that are Jacobson \'etendues (over $\Sets$).
\begin{enumerate}
\item $\PSh(G)$ for $G$ a group is a Jacobson \'etendue. Indeed, we can take $E=G$ with its standard right $G$-action, and then $\PSh(G)/G \simeq \Sets$.
\item The J\'onsson--Tarski topos $\mathcal{J}$ is a Jacobson \'etendue. Here we can take $E$ to be the free J\'onsson-Tarski algebra on one generator, and then $\mathcal{J}/E \simeq \Sh(X)$, for $X$ the Cantor space, see \cite[Proposition 8.5.2]{bunge-funk}. The Cantor space is Hausdorff, so it is in particular Jacobson.
\end{enumerate}
\end{example}

\begin{lemma} \label{lmm:EILC-is-local}
Let $\Ecal$ be an elementary topos and let $E$ be a well-supported object of $\Ecal$. If $\Ecal/E$ is EILC, then $\Ecal$ is EILC as well.
\end{lemma}
\begin{proof}
Let $f : \Fcal \to \Ecal$ be an essential geometric morphism. The slice 
\begin{equation*}
f/E : \Fcal/f^*E \to \Ecal/E
\end{equation*}
is then also essential \cite[Lemma 5.2]{lawvere-menni}. Because $\Ecal/E$ is EILC, we see that $f/E$ is locally connected. Note that $f/E$ is the base change of $f$ along the \'etale geometric morphism $\pi : \Ecal/E \to \Ecal$. Because $E$ is well-supported, the \'etale geometric morphism is a surjection. It follows that $f$ is locally connected as well, because local connectedness can be checked after base change along an \'etale surjection, see for example \cite[Corollary C3.3.2(iv)]{johnstone-elephant}.
\end{proof}

Let $\Ecal$ be an elementary topos with a natural number object. If $\Ecal$ is a Jacobson \'etendue over an EILC base topos $\Scal$, then we can take a well-supported object $E$ in $\Ecal$ such that $\Ecal/E$ is a Jacobson space over $\Scal$. By Theorem \ref{thm:jacobson-space-is-EILC}, it follows that $\Ecal/E$ is EILC. So by applying Lemma \ref{lmm:EILC-is-local}, we find that $\Ecal$ is EILC. In summary:

\begin{corollary} \label{cor:jacobson-is-eilc}
Let $\Ecal$ be an elementary topos with a natural numbers object. If $\Ecal$ is a Jacobson \'etendue over an EILC base topos $\Scal$, then $\Ecal$ is EILC as well.
\end{corollary}

In particular, $\PSh(G)$ is EILC for any group $G$, and the J\'onsson--Tarski topos is EILC.

\section{Boolean \'etendues and compact topological groups}
\label{sec:boolean-etendues}

In this section we restrict to Grothendieck toposes, i.e.\ toposes bounded over the topos of sets. We will show that both Boolean \'etendues and classifying toposes of compact topological groups are EILC. Afterwards, we show that a presheaf topos $\PSh(\Ccal)$ is EILC if and only if $\Ccal$ is a groupoid. 

For both Boolean \'etendues and classifying toposes of compact topological groups, the argument can be simplified using the following lemma:

\begin{lemma} \label{lmm:EILC-for-stable-families}
Let $\mathcal{A}$ be a family of toposes with the following properties:
\begin{enumerate}
\item if $\Ecal$ is in $\mathcal{A}$, then also $\Ecal/E$ is in $\mathcal{A}$, for any object $E$ in $\Ecal$;
\item for any essential geometric morphism $f : \Fcal \to \Ecal$, with $\Ecal$ in $\mathcal{A}$, the inverse image functor $f^*$ is cartesian closed.
\end{enumerate}
Then all toposes in $\mathcal{A}$ are EILC.
\end{lemma}
\begin{proof}
For $\Ecal$ in $\Acal$, we have to show that any essential geometric morphism $f : \Fcal \to \Ecal$ is locally connected. This follows from the following characterization of local connectedness: $f$ is locally connected if and only if its slice $f/E : \Fcal/f^*E \to \Ecal/E$ has cartesian closed inverse image functor, for all objects $E$ in $\Ecal$ \cite[Proposition C3.3.1]{johnstone-elephant}. If $f$ is essential, then each slice $f/E$ is again essential \cite[Lemma 5.2]{lawvere-menni}, and by (1) its codomain $\Ecal/E$ is in $\Acal$. So by (2) $f/E$ has cartesian closed inverse image functor.
\end{proof}

A \textbf{Boolean \'etendue} is a topos that is both Boolean and an \'etendue. Note that if $E$ is a well-supported object of a topos $\Ecal$, then $\Ecal$ is Boolean if and only if $\Ecal/E$ is Boolean. So we can alternatively define a Grothendieck topos $\Ecal$ to be a Boolean \'etendue if there is a well-supported object $E$ in $\Ecal$ and a Boolean locale $Y$ with $\Ecal/E \simeq \Sh(Y)$.

We now show that any Grothendieck topos that is a Boolean \'etendues, is EILC. The proof is inspired by a related argument by Mat\'ias Menni, in his proof that for an arbitrary Boolean topos $\Ecal$, a connected, essential geometric morphism $f : \Fcal \to \Ecal$ is locally connected as soon as $f_!$ preserves finite products.

\begin{proposition} \label{prop:boolean-etendues-are-EILC}
Let $\Ecal$ be a Grothendieck topos. If $\Ecal$ is a Boolean \'etendue, then it is EILC.
\end{proposition}
\begin{proof}
By Lemma \ref{lmm:EILC-is-local} it is enough to prove that localic Boolean Grothendieck toposes are EILC. Further, by applying Lemma \ref{lmm:EILC-for-stable-families} for $\Acal$ the family of localic Boolean Grothendieck toposes, it is enough to show that any essential geometric morphism $f : \Fcal \to \Ecal$ has cartesian closed inverse image functor, for $\Ecal$ a localic Boolean Grothendieck topos. This is equivalent to showing that for any objects $X$ in $\Fcal$ and $A$ in $\Ecal$ the natural map
\begin{equation*}
\theta_{X,A} : f_!(X \times f^*A) \to f_!(X) \times A
\end{equation*}
is an isomorphism. Because $\Ecal$ is a localic Grothendieck topos, $A$ can be written as a colimit of subterminal objects. So it is enough to prove that $\theta_{X,A}$ is an isomorphism in the special case that $A$ is subterminal (colimits are preserved by $f_!$ and $f^*$ and are stable under pullbacks). 

Now take the complement $A'$ of $A$, so $1 = A \sqcup A'$. Since $\theta_{X,1}$ is trivially an isomorphism, its restrictions $\theta_{X,A}$ and $\theta_{X,A'}$ are isomorphisms as well. Alternatively, we can argue that in the pullback diagram
\begin{equation*}
\begin{tikzcd}
\Fcal/f^*A \ar[r,"{\tilde{\pi}}"] \ar[d,"{f/A}"'] & \Fcal \ar[d,"{f}"] \\
\Ecal/A \ar[r,"{\pi}"'] & \Ecal
\end{tikzcd}
\end{equation*}
the Beck--Chevalley condition holds, because $\pi : \Ecal/A \to \Ecal$ is a closed inclusion, in particular bounded and tidy, so Proposition \ref{prop:locally-connected-tidy}(3) applies. The Beck--Chevalley condition in this case says precisely that $\theta_{X,A}$ is an isomorphism for each object $X$ in $\Fcal$. 
\end{proof}

For a topological group $G$, its classifying topos $\mathbf{Cont}(G)$ is the topos of sets equipped with a continuous action of $G$ (where the sets are seen as topological spaces with the discrete topology). We will now show that $\mathbf{Cont}(G)$ is EILC if the topological group $G$ is compact. We would not gain any generality by considering compact \textit{localic} groups, because over $\Sets$ any compact localic group has enough points \cite[Remarks 5.3.14(b)]{johnstone-elephant}.

We will simplify our argument by applying Lemma \ref{lmm:EILC-for-stable-families}. In order for this to work, we need to consider more generally toposes of the form $\bigsqcup_{i \in I} \mathbf{Cont}(G_i)$, for $(G_i)_{i \in I}$ a family of compact topological groups (the disjoint union is computed in the category of Grothendieck toposes). An object in $\bigsqcup_{i \in I} \mathbf{Cont}(G_i)$ is a family $(A_i)_{i \in I}$ with each $A_i$ an object in $\mathbf{Cont}(G_i)$. We claim that if $\Ecal$ is of the form $\bigsqcup_{i \in I} \mathbf{Cont}(G_i)$ for some family of compact topological groups $(G_i)_{i \in I}$, then $\Ecal/A$ is again of the same form, for each object $A$ in $\Ecal$.

Indeed, if $A = (A_i)_{i \in I}$ is an object in $\Ecal \simeq \bigsqcup_{i \in I} \mathbf{Cont}(G_i)$, then 
\begin{equation*}
\Ecal/A ~\simeq~ \bigsqcup_{i \in I} \mathbf{Cont}(G_i)/A_i. 
\end{equation*}
We can write each $A_i$ as a disjoint union of orbits $A_i \cong \bigsqcup_{j \in J_i} G_i/H_{ij}$, with $H_{ij} \subseteq G_i$ an open subgroup, for each $j \in J_i$. Now using the equivalence $\mathbf{Cont}(G_i)/(G_i/H_{ij}) \simeq \mathbf{Cont}(H_{ij})$, we find that
\begin{equation*}
\Ecal/A ~\simeq~ \bigsqcup_{i \in I} \bigsqcup_{j \in J_i} \mathbf{Cont}(H_{ij}).
\end{equation*}
Note that each group $H_{ij}$ is again compact, because it is an open subgroup of the compact topological group $G_i$ (and open subgroups are closed). So $\Ecal/A$ is of the same form.

For an object $A = (A_i)_{i \in I}$ in $\Ecal$, we would now like to determine when the corresponding \'etale geometric morphism $\Ecal/A \to \Ecal$ is tidy. Each topos $\mathbf{Cont}(G_i)$ has a canonical point $\Sets \to \mathbf{Cont}(G_i)$ which is an open surjection, so taking the disjoint union of these points gives an open surjection of the form
\begin{equation*}
\xi ~:~ \bigsqcup_{i \in I} \Sets \longrightarrow \Ecal.
\end{equation*}
The inverse image functor $\xi^*$ is the forgetful functor, sending a family $(A_i)_{i \in I}$ to the same family $(A_i)_{i \in I}$, but this time each $A_i$ is seen only as a set. The property of being tidy can be checked after base change along the open surjection $\xi$, so $\Ecal/A \to \Ecal$ is tidy if and only if
\begin{equation*}
\bigsqcup_{i \in I} \Sets/A_i \longrightarrow \bigsqcup_{i \in I} \Sets
\end{equation*}
is tidy. We conclude that $\Ecal/A \to \Ecal$ is tidy if and only if the underlying set of $A_i$ is finite, for all $i \in I$. This will be relevant in the next result, because of the relation between tidy geometric morphisms and the Beck--Chevalley condition.

\begin{proposition} \label{prop:compact-topological-group-EILC}
Let $\{G_i\}_{i \in I}$ be a family of compact topological groups. Then the topos $\bigsqcup_{i \in I} \mathbf{Cont}(G_i)$ is EILC.
\end{proposition}
\begin{proof}
We write $\Ecal \simeq \bigsqcup_{i \in I} \mathbf{Cont}(G_i)$. Let $f : \Fcal \to \Ecal$ be an essential geometric morphism. We have to show that $f$ is locally connected. Applying Lemma \ref{lmm:EILC-for-stable-families}, it is enough to show that $f^*$ is cartesian closed. In other words, we have to show that the natural map $\theta_{X,A} : f_!(X \times f^*A) \to f_!(X) \times A$ is an isomorphism, for $X$ in $\Fcal$ and $A$ in $\Ecal$. Because $\Ecal$ is locally connected, we can write $A$ as a coproduct of connected objects. Coproducts are pullback-stable and preserved by $f^*$ and $f_!$, so we can reduce to the case where $A$ is connected. Note that if $A$ corresponds to the family $(A_i)_{i \in I}$ with each $A_i$ an object in $\mathbf{Cont}(G_i)$, then $A$ being connected implies that there is an index $i_0 \in I$ such that $A_{i_0}$ is a single $G_i$-orbit, and $A_j = \varnothing$ for all $j \neq i_0$. Using compactness of $G_{i_0}$, it follows that the underlying set of $A_{i_0}$ is finite. By the discussion above, we then have that the \'etale geometric morphism $\Ecal/A \to \Ecal$ is tidy. In particular, the pullback square
\begin{equation*}
\begin{tikzcd}
\Fcal/f^*A \ar[r,"{\tilde{\pi}}"] \ar[d,"{f/A}"'] & \Fcal \ar[d,"{f}"] \\
\Ecal/A \ar[r,"{\pi}"] & \Ecal 
\end{tikzcd}
\end{equation*}
satisfies the Beck--Chevalley condition, see Proposition \ref{prop:locally-connected-tidy}(3). But this coincides precisely with the statement that the natural map $f_!(X \times f^*A) \to f_!(X) \times A$ is an isomorphism, which is what we wanted to prove.
\end{proof}

For presheaf toposes, we have a jointly surjective family of essential points, and these points are typically not locally connected. As a result, presheaf toposes will usually not be EILC. More precisely:

\begin{proposition} \label{prop:characterization-presheaf-topos-EILC}
Let $\Ccal$ be a small category. Then $\PSh(\Ccal)$ is EILC if and only if $\Ccal$ is a groupoid.
\end{proposition}
\begin{proof}
If $\Ccal$ is a groupoid, then $\PSh(\Ccal)$ is a Boolean \'etendue, so we can use Proposition \ref{prop:boolean-etendues-are-EILC} to conclude that $\PSh(\Ccal)$ is EILC.

Conversely, suppose that $\PSh(\Ccal)$ is EILC. Each object $C$ in $\Ccal$ determines an essential point $p : \Sets \to \PSh(\Ccal)$ with $p_!(1) \simeq \mathbf{y}C$, $\mathbf{y}$ the Yoneda embedding. Because of the EILC property, $p$ has to be locally connected. We can then factorize $p$ as a connected, locally connected geometric morphism, followed by an \'etale geometric morphism. However, because the domain topos is $\Sets$, the connected part is trivial, so $p$ is \'etale. It follows from $p_!(1)\simeq \mathbf{y}C$ that we then have $\PSh(\Ccal/C)\simeq \Sets$. This is only possible if $\Ccal/C$ has only one object up to isomorphism, or in other words any morphism $D \to C$ is necessarily an isomorphism. Because $C$ was arbitrary, we conclude that $\Ccal$ is a groupoid.
\end{proof}

More generally, we could consider the presheaf toposes $\PSh_\Scal(\Ccal)$, for $\Ccal$ an internal category in an arbitrary EILC topos $\Scal$. However, it is not known at the moment to the present author whether the analogue of Proposition \ref{prop:characterization-presheaf-topos-EILC} would still hold. 

\section{A weaker property: CILC toposes}

We will say that an elementary topos $\Ecal$ is \textbf{CILC} if any geometric morphism $f : \Fcal \to \Ecal$ such that $f^*$ is cartesian closed (i.e.\ preserves exponential objects) is automatically locally connected. 

Note that by a result of Barr and Par\'e \cite[Theorem 2]{barr-pare}, if $f^*$ is cartesian closed, then $f$ is essential. I thank Thomas Streicher for pointing out this result to me. So all EILC toposes are in particular CILC. The converse does not hold; we will construct some counterexamples below.

\begin{proposition} \label{prop:sufficient-condition-CILC}
Let $\Ecal$ be an elementary topos with a natural number object, and suppose that there is a jointly surjective family $\{p_i : \Ecal_i \to \Ecal \}_{i \in I}$, such that each $p_i$ can be factored as a closed inclusion followed by an \'etale geometric morphism, and such that each $\Ecal_i$ is EILC. Then $\Ecal$ is CILC.
\end{proposition}
\begin{proof}
Let $f : \Fcal \to \Ecal$ be a geometric morphism with $f^*$ cartesian closed. For each $i \in I$, the geometric morphism $p_i$ factors as
\begin{equation*}
\Ecal_i \stackrel{j}{\longrightarrow} \Ecal/E \stackrel{\pi_E}{\longrightarrow} \Ecal
\end{equation*}
with $j$ a closed inclusion and $\pi_E : \Ecal/E \to \Ecal$ the \'etale geometric morphism corresponding to an object $E$ in $\Ecal$.

Now consider the corresponding pullback squares
\begin{equation*}
\begin{tikzcd}
\Fcal_i \ar[r,"{\tilde{j}}"] \ar[d,"{f_i}"'] & \Fcal/f^*(E) \ar[r,"{\tilde{\pi}_E}"] \ar[d,"{f/E}"] & \Fcal \ar[d,"{f}"] \\
\Ecal_i \ar[r,"{j}"'] & \Ecal/E \ar[r,"{\pi_E}"'] & \Ecal
\end{tikzcd}\quad.
\end{equation*}
Note that because $j$ and $\pi_E$ are localic, they are bounded. Further, since $j$ is a closed inclusion, it is in particular tidy, see \cite[Chapter III, Corollary 5.8]{moerdijk-vermeulen}. So using Proposition \ref{prop:locally-connected-tidy}(3), we find $f_i \bcp_j (f/E)$. Because $(f/E)$ is essential, it then follows from Proposition \ref{prop:stability} that $f_i$ is essential as well. But then $f_i$ is locally connected, because $\Ecal_i$ is EILC. 

Moreover, it follows from cartesian closedness of $f^*$ that $(f/E)\bcp_{\pi_E} f$, see Proposition \ref{prop:characterization-cartesian-closed}. Using Proposition \ref{prop:transitivity} and $p_i \simeq \pi_E \circ j$ we conclude that $f_i \bcp_{p_i} f$. So for each $i \in I$, we find that $f$ is locally connected at $p_i$. Because the family $\{p_i : \Ecal_i \to \Ecal \}$ is jointly surjective, we then conclude that $f$ is locally connected, see Proposition \ref{prop:descent}.
\end{proof}

The assumption that $\Ecal$ has a natural number object is relevant when we apply Proposition \ref{prop:locally-connected-tidy}(3). Alternatively, we could use Proposition \ref{prop:locally-connected-tidy}(4), but then we can only conclude that $\Ecal$ has the property that any \textit{bounded} geometric morphism $f : \Fcal \to \Ecal$, with $f^*$ cartesian closed, is locally connected.

\begin{definition}
A geometric morphism $f : \Ecal \to \Scal$ will be called \textbf{weakly Jacobson} if there is a jointly surjective family of points $\{p_i : \Scal \to \Ecal \}_{i\in I}$, such that each $p_i$ can be factored as a closed inclusion followed by an \'etale geometric morphism.
\end{definition}

\begin{theorem}
Let $f : \Ecal \to \Scal$ be a geometric morphism. Suppose that $\Ecal$ has a natural numbers object. If $f$ is weakly Jacobson and $\Scal$ is EILC, then $\Ecal$ is CILC.
\end{theorem}
\begin{proof}
If $f$ is weakly Jacobson, then by definition there is a jointly surjective family of $\{ p_i : \Scal \to \Ecal \}_{i \in I}$ such that each $p_i$ can be factored as a closed inclusion followed by an \'etale geometric morphism. If moreover $\Scal$ is EILC, then Proposition \ref{prop:sufficient-condition-CILC} applies, and we conclude that $\Ecal$ is CILC. 
\end{proof}

We will restrict to Grothendieck toposes in the remainder of this section. We first characterize the topological spaces $X$ such that $\Sh(X)$ is weakly Jacobson (over the topos of sets). 

\begin{proposition} \label{prop:characterization-topological-space-weakly-jacobson}
Let $X$ be a topological space, and let $X_{\mathrm{lc}} \subseteq X$ be the subset of locally closed points. Then $\Sh(X)$ is weakly Jacobson if and only if $U \cap X_{\mathrm{lc}} = V \cap X_{\mathrm{lc}}$ implies $U = V$, for all open subsets $U,V \subseteq X$.
\end{proposition}
\begin{proof}
The locally closed points of $X$ are precisely the points that are open in their closure. So if $x \in X$ is a locally closed point, then there is an open set $U \subseteq X$ such that $U \cap \overline{\{x\}} = \{x\}$. In this situation, $x$ is the only point that can distinguish between the open sets $W$ and $W \cup U$, for $W = X-\overline{\{x\}}$. This implies that a jointly surjective family of points $\{ p_i : \Sets \to \Sh(X) \}$ will necessarily contain all points $p_x : \Sets \to \Sh(X)$ corresponding to locally closed points $x \in X_{\mathrm{lc}} \subseteq X$.

In particular, let $\tilde{X}$ be the sobrification of $X$. Then the elements of $X$ determine a jointly surjective family of points for $\Sh(\tilde{X})$, and by the above this means that all locally closed points of $\tilde{X}$ are also contained in $X$. So we can assume without loss of generality that $X$ is sober, i.e.\ that the correspondence between elements of $X$ and topos-theoretic points $\Sets \to \Sh(X)$ (up to isomorphism) is bijective.

Now suppose that $\Sh(X)$ is weakly Jacobson, or in other words that there exists a jointly surjective family of points $\{ p_i : \Sets \to \Sh(X) \}_{i \in I}$, such that each $p_i$ can be factored as a closed inclusion followed by an \'etale geometric morphism. Let $x_i \in X$ be the element corresponding to $p_i$. The embedding $\{x_i\} \subseteq X$ can then be factored as a closed inclusion $\{x_i\} \subseteq E$ followed by a local homeomorphism $\pi: E \to X$. Take an open set $U$ containing $x_i$ such that the restriction of $\pi$ defines an homeomorphism from $U$ to the open set $\pi(U) \subseteq X$. Then $\{x_i\} \subseteq X$ factors as a closed inclusion $\{x_i\} \subseteq \pi(U)$ followed by an open inclusion $\pi(U) \subseteq X$. So each $x_i$ is a locally closed point. As a result, the locally closed points of $X$ form a jointly surjective family, i.e.\ if two open subsets $U,V$ contain the same locally closed points, then $U=V$.

Conversely, suppose that the locally closed points form a jointly surjective family. For each locally closed point $x \in X$, we can write $\{x\} = U \cap V$ with $U$ open and $V$ closed. But then the inclusion $\{x\} \subseteq X$ factorizes as a closed inclusion $\{x\} \subseteq U$ followed by the open inclusion $U \subseteq X$, which is in particular a local homeomorphism. But then $\Sh(X)$ is weakly Jacobson.
\end{proof}

\begin{example}
The Sierpinski space is given by $\mathbb{S} = \{m,g\}$ with as open sets $\varnothing$, $\{g\}$ and $\{g,m\}$. Both points are locally closed ($g$ is open and $m$ is closed). It then follows by Proposition \ref{prop:characterization-topological-space-weakly-jacobson} that $\Sh(\mathbb{S})$ is weakly Jacobson. As a result, $\Sh(\mathbb{S})$ is CILC.

Note that $\Sh(\SS) \simeq \PSh(\Ccal)$, for $\Ccal$ the category with two objects $A$ and $B$ and a single non-identity morphism $A \to B$. So from Proposition \ref{prop:characterization-presheaf-topos-EILC} it follows that $\Sh(\SS)$ is not EILC.
\end{example}

\begin{example}
In some topological spaces, none of the points are locally closed. Take for example the set $X \subset \mathcal{P}(\NN)$ of infinite subsets of natural numbers, with as topology the smallest topology such that the sets
\begin{equation*}
U_n = \{ V : V \ni n  \} \subseteq X
\end{equation*}
are open. Then none of the points of $X$ are locally closed, so $\Sh(X)$ is not weakly Jacobson. However, it is not known to the author whether $\Sh(X)$ is CILC or even EILC.
\end{example}

We also want to give a criterion for when a presheaf topos is weakly Jacobson (over the topos of sets). We first need the following lemma:

\begin{lemma} \label{lmm:local-with-closed-center}
Let $\Ccal$ be a small category with a terminal object. Then $\PSh(\Ccal)$ is local. Moreover, its center is a closed inclusion if and only if the terminal object in $\Ccal$ is strict.
\end{lemma}
\begin{proof}
If $\Ccal$ has a terminal object, then $\PSh(\Ccal)$ is local, see \cite[Examples C3.6.3(b)]{johnstone-elephant}. The center of a local geometric morphism is an inclusion, and in this case it agrees with the essential point $p : \Sets \to \PSh(\Ccal)$ corresponding to the terminal object in $\Ccal$.

From \cite[Lemma C3.2.4]{johnstone-elephant} it then follows that $p$ is closed if and only if every morphism $b : 1 \to C$ in $\Ccal$ admits a right inverse $r : C \to 1$. Whenever such a right inverse $r$ exists, it must also be a left inverse, because $rb$ is an endomorphism of the terminal object. So we find that $p$ is closed if and only if every morphism $1 \to C$ is an isomorphism, or in other words if and only if the terminal object is strict. 
\end{proof}

Note that the argument in \cite[Lemma C3.2.4]{johnstone-elephant} is not constructive. So our argument here does not generalize to presheaf toposes over an arbitrary base topos.

\begin{proposition} \label{prop:characterization-presheaf-topos-weakly-jacobson}
Let $\Ccal$ be a small category. If every morphism in $\Ccal$ admitting a right inverse is an isomorphism, then $\PSh(\Ccal)$ is weakly Jacobson. Conversely, if $\PSh(\Ccal)$ is weakly Jacobson, then there is a small category $\Ccal'$, with $\PSh(\Ccal)\simeq \PSh(\Ccal')$, such that every morphism in $\Ccal'$ admitting a right inverse is an isomorphism.
\end{proposition}
\begin{proof}
For every object $C$, we can consider the corresponding point 
\begin{equation*}
p_C : \Sets \to \PSh(\Ccal). 
\end{equation*}
This point factors as $j : \Sets \to \PSh(\Ccal/C)$ followed by $\pi: \PSh(\Ccal/C) \to \PSh(\Ccal)$. If every morphism $f : D \to C$ that admits a right inverse is an isomorphism, then the terminal object in $\Ccal/C$ is strict. But then using Lemma \ref{lmm:local-with-closed-center} we see that $j$ is a closed inclusion. In the definition of weakly Jacobson, we can now take the family $\{p_C : \Sets \to \PSh(\Ccal) \}_C$, with $C$ going over the objects of $\Ccal$, to conclude that $\PSh(\Ccal)$ is indeed weakly Jacobson.

Conversely, let $\{ p_i : \Sets \to \PSh(\Ccal) \}_{i \in I}$ be a jointly surjective family, such that each point $p_i$ can be factored as a closed inclusion $j : \Sets \to \PSh(\Dcal)$ followed by an \'etale geometric morphism $\pi: \PSh(\Dcal) \to \PSh(\Ccal)$. Because $j$ is a closed inclusion, it must be essential. Indeed, otherwise all essential points would be contained in the complement of the subtopos defined by $j$, and because the essential points form a jointly surjective family, this means that the complement of $j$ is the full topos $\PSh(\Dcal)$, a contradiction. As a result, we know that $p_i$ is essential, for all $i \in I$. The family $\{ p_i : \Sets \to \PSh(\Ccal) \}_{i \in I}$ is jointly surjective, so we can find a small category $\Ccal'$, with $\PSh(\Ccal')\simeq\PSh(\Ccal)$, such that 
\begin{equation*}
\{ p_i : \Sets \to \PSh(\Ccal') \}_{i \in I} = \{ p_C : \Sets \to \PSh(\Ccal') \}_{C \in \mathbf{Ob}(\Ccal')}, 
\end{equation*}
for $p_C : \Sets \to \Ccal'$ the essential geometric morphism associated to $C$ (points are considered up to isomorphism).

For each object $C$ in $\Ccal'$, we can now factorize $p_C$ as a closed inclusion $j : \Sets \to \PSh(\Dcal)$ followed by the \'etale geometric morphism $\pi: \PSh(\Dcal) \to \PSh(\Ccal')$, as above. We further have a different factorization of $p_C$ as an inclusion $j' : \Sets \to \PSh(\Ccal'/C)$ followed by an \'etale geometric morphism $\pi' : \PSh(\Ccal'/C) \to \PSh(\Ccal')$. In fact, the latter is precisely the (terminal-connected, \'etale) factorization as described by Caramello in \cite[Section 4.7]{caramello-denseness}. The geometric morphism $j'$ is the center of the local topos $\PSh(\Ccal'/C)$. We want to show that $j'$ is closed, because then we can apply Lemma \ref{lmm:local-with-closed-center}.

We apply the (terminal-connected, \'etale) factorization to the closed inclusion $j : \Sets \to \PSh(\Dcal)$. By uniqueness of (terminal-connected, \'etale) factorizations \cite[Proposition 4.62]{caramello-denseness}, this factorization must be given by $j' : \Sets \to \PSh(\Ccal'/C)$ followed by an \'etale geometric morphism $\pi'' : \PSh(\Ccal'/C) \to \PSh(\Dcal)$. Now consider the pullback diagram
\begin{equation*}
\begin{tikzcd}
\Sets/A \ar[r,"{\tilde{j}}"] \ar[d,"{\gamma}"'] & \PSh(\Ccal'/C) \ar[d,"{\pi''}"] \\
\Sets \ar[r,"{j}"] & \PSh(\Dcal)
\end{tikzcd}
\end{equation*}
and note that $j' : \Sets \to \PSh(\Ccal'/C)$ can be written as $j' = \tilde{j}\circ s$, for $s$ a section of $\gamma$ (see \cite[Expos\'e IV, Proposition 5.12]{sga4-1}). The geometric morphism $\tilde{j}$ is the pullback of the closed inclusion $j$, so it is itself a closed inclusion. Further, any section of $\gamma$ is also a closed inclusion (a discrete topological space has closed points). It follows that $j'$ is a closed inclusion. By Lemma \ref{lmm:local-with-closed-center}, this implies that the terminal object in $\Ccal'/C$ is strict. In the above, the object $C$ was arbitrary, so $\Ccal'/C$ has a strict terminal object for all objects $C$ in $\Ccal'$. In other words, if an arbitrary morphism in $\Ccal'$ has a right inverse, then it must be an isomorphism.
\end{proof}

\begin{example} \ 
\begin{enumerate}
\item The topos of directed graphs is the topos of presheaves on a category $\Ccal$ with two objects $V$ and $E$, and as morphisms the two identity morphisms and $s,t : V \to E$. The only morphisms in $\Ccal$ that admit a right inverse, are the identity morphisms. So the topos of directed graphs is weakly Jacobson, in particular CILC.
\item Let $M$ be a monoid such that every right-invertible element is (two-sided) invertible. Then $\PSh(M)$ is weakly Jacobson, in particular CILC.
\item Let $N$ be the monoid of natural numbers (with zero) under multiplication. Consider the category $\Ccal$ with as objects the left $N$-set $N$, with the action given by multiplication, and the terminal left $N$-set $1$. As morphisms, we take the morphisms of left $N$-sets. The unique morphism $N \to 1$ in $\Ccal$ then admits a right inverse. However, because $1$ is a retract of $N$ in $\Ccal$, we have $\PSh(\Ccal) \simeq \PSh(N)$. Moreover, in $N$ there is only one element that admits a right inverse, namely the identity. So $\PSh(\Ccal) \simeq \PSh(N)$ is weakly Jacobson.
\item Let $M$ be a monoid containing a right-invertible element that is not invertible. Suppose that $\PSh(M) \simeq \PSh(M')$ for a different monoid $M'$. Then $M'$ again contains a right-invertible element that is not invertible, otherwise the Morita equivalence $\PSh(M) \simeq \PSh(M')$ would imply that $M \cong M'$, see for example \cite[Corollary 7.2(3)]{mr-monoid-actions}. If, more generally, $\Ccal'$ is a category with $\PSh(M) \simeq \PSh(\Ccal')$, then there is an object in $\Ccal'$ that is a generator, in the sense that its endomorphism monoid $M'$ satisfies $\PSh(M)\simeq \PSh(M')$. We conclude that in $\Ccal'$ there is a right-invertible morphism that is not invertible. It then follows from Proposition \ref{prop:characterization-presheaf-topos-weakly-jacobson} that $\PSh(M)$ is not weakly Jacobson.
\end{enumerate}
\end{example}

Finally, we also show that every Boolean elementary topos is CILC. 

\begin{theorem} \label{thm:boolean-CILC}
Every Boolean elementary topos is CILC.
\end{theorem}
\begin{proof}
Let $\Ecal$ be a Boolean elementary topos, and let $f : \Fcal \to \Ecal$ be a geometric morphism such that the inverse image functor $f^*$ is cartesian closed. We need to show that for every morphism $\phi: A \to B$, the associated pullback square
\begin{equation} \label{eq:square-for-phi}
\begin{tikzcd}
\Fcal/f^*A \ar[d,"{f/A}"'] \ar[r,"{\tilde{\pi}}"] & \Fcal/f^*B \ar[d,"{f/B}"] \\
\Ecal/A \ar[r,"{\pi}"] & \Ecal/B
\end{tikzcd}
\end{equation}
satisfies the Beck--Chevalley condition $(f/A) \bcp_\pi (f/B)$.

There are two situations in which we know this Beck--Chevalley condition is satisfied. First of all, if $\phi$ is a monomorphism, then $\pi$ is an inclusion, and because $\Ecal/B$ is Boolean, it must be a closed inclusion. So it is bounded and tidy, and then the Beck--Chevalley condition is automatically satisfied, see Proposition \ref{prop:locally-connected-tidy}(3). Note however that in order to apply Proposition \ref{prop:locally-connected-tidy}(3), we need that $\Ecal$ has a natural numbers object. To avoid this extra assumption, we give an alternative argument. Consider the natural map
\begin{equation*}
\theta_{X,B,A} ~:~ f_!(X \times_{f^*B} f^*A) \longrightarrow f_!(X) \times_B A.
\end{equation*}
If $\phi : A \to B$ is an inclusion, then because $\Ecal$ is Boolean, we can take a complement $A'$ of $A$. The natural map $\theta_{X,B,B}$ associated to the identity $B \to B$ is trivially an isomorphism, so its restrictions $\theta_{X,B,A}$ and $\theta_{X,B,A'}$ are isomorphisms as well. As a result, the diagram \eqref{eq:square-for-phi} satisfies $(f/A) \bcp_\pi (f/B)$ as soon as $\phi$ is injective.

A second situation when the Beck--Chevalley condition is satisfied is when $X = 1$, because in this case the Beck--Chevalley condition follows from $f^*$ being cartesian closed, see Proposition \ref{prop:characterization-cartesian-closed}. More generally, if $\phi$ is of the form $\pi_B : Y \times B \to B$ (projection on the second component), then the corresponding pullback square is a slice over $B$ of the pullback square for $\phi : Y \to 1$, so it is again a Beck--Chevalley square, see \cite[Lemma A4.1.16]{johnstone-elephant}.

In general, we can factor $\phi : Y \to X$ as the inclusion $j : Y \to Y \times X$, $j = (\mathrm{id}_Y,\phi)$ followed by the projection $\pi_X : Y \times X \to X$. But then the square in \eqref{eq:square-for-phi} satisfies the Beck--Chevalley condition by applying transitivity, see Proposition \ref{prop:transitivity}.
\end{proof}

\begin{remark}
Theorem \ref{thm:boolean-CILC} extends an earlier result by Mat\'ias Menni, who showed that a connected, essential geometric morphism $f : \Fcal \to \Ecal$, with $\Ecal$ a Boolean topos, is locally connected as soon as $f_!$ preserves finite products. Note that if $f$ is connected and $f_!$ preserves finite products, then $f^*$ is cartesian closed, see \cite[Proposition A4.3.1]{johnstone-elephant}.
\end{remark}

\section*{Acknowledgements}

I would like to thank Thomas Streicher, Mat\'ias Menni and Morgan Rogers for interesting discussions leading to this article, and for helpful comments on draft versions. The author is a postdoctoral fellow of the Research Foundation Flanders (file number 1276521N).

\providecommand{\bysame}{\leavevmode\hbox to3em{\hrulefill}\thinspace}
\providecommand{\MR}{\relax\ifhmode\unskip\space\fi MR }
\providecommand{\MRhref}[2]{%
  \href{http://www.ams.org/mathscinet-getitem?mr=#1}{#2}
}
\providecommand{\href}[2]{#2}

\end{document}